\documentclass[reqno]{amsart}%
\usepackage{amsmath}
\usepackage{amsfonts}%
\usepackage{amssymb}%
\usepackage{graphicx}
\providecommand{\U}[1]{\protect\rule{.1in}{.1in}}
\newtheorem{theorem}{Theorem}
\theoremstyle{plain}

\newtheorem{lemma}{Lemma}

\numberwithin{equation}{section}
\begin{document}
\title[ ]{Approximation by a generalization of the Jakimovski -Leviatan operators }
\author{Didem AYDIN\ ARI}
\address{K\i r\i kkale University }
\email{didemaydn@hotmail.com}
\urladdr{}
\author{Sevilay KIRCI\ SERENBAY}
\curraddr{Ba\c{s}kent University}
\email{sevilaykirci@gmail.com}
\urladdr{}
\thanks{}
\thanks{}
\thanks{This paper is in final form and no version of it will be submitted for
publication elsewhere.}
\date{2015}
\subjclass[2000]{ }

\begin{abstract}
In this paper, we introduce a Kantorovich type generalization of
\ Jakimovski-Leviatan operators constructed by A. Jakimovski and D. Leviatan
(1969) and the theorems on convergence and the degrree of convergence are
established. Furthermore, we study the convergence of these operators in a
weighted space of functions on a positive semi-axis.

\end{abstract}
\maketitle

\section{Introduction}

\bigskip

In approximation theory, Sz\`{a}sz type operators and Chlodowsky type
generalizations of these operators have been studied intensively [see
\cite{jaki}, \cite{C1}, \cite{nurh}, \cite{ibrahim}, \cite{gulen},
\cite{agar}, \cite{serhan}, \cite{chlo}, \cite{didem} and many others]. Also
orthogonal polynomials are important area of mathematical analysis,
mathematical and theoretical physics. In mathematical analysis and in the
positive approximation processes, the notion of orthogonal polynomials
seldomly appears. Cheney and Sharma \cite{Cheney} established an operator
\begin{equation}
P_{n}(f;x)=\left(  1-x\right)  ^{n+1}\exp\left(  \frac{tx}{1-x}\right)
{\displaystyle\sum\limits_{k=0}^{\infty}}
f\left(  \frac{k}{k+n}\right)  L_{k}^{\left(  n\right)  }(t)x^{k}%
\end{equation}
where $t\leq0$ and $L_{k}^{\left(  n\right)  }$ denotes the Laguerre
polynomials. For the special case $t=0$, the operators given by (\ref{1.1})
reduce to the well-known Meyer-K\"{o}nig and Zeller operators \cite{Mkz}.

In view of the relation between orthogonal polynomials and positive linear
operators have been investigated by many researchers (see \cite{serhan}%
,\cite{didem}).

One of them is Jakimovski and Leviatan 's study. In 1969, authors introduced
Favard-Sz\`{a}sz type operators $P_{n},$ by using Appell polynomials are given with

$g(u)=\sum\limits_{n=0}^{\infty}a_{n}u^{n},g(1)\neq1$ be an analytic function
in the disk $\left\vert u<r\right\vert $ $(r>1)$ and $p_{k}(x)=\sum
\limits_{i=0}^{k}a_{i}\frac{x^{k-i}}{(k-i)!}$, $(k\in%
\mathbb{N}
)$ be the Appell polynomials defined by the identity%
\begin{equation}
g(u)e^{ux}\equiv\sum\limits_{k=0}^{\infty}p_{k}(x)u^{k}.\label{1.1}%
\end{equation}

\bigskip Let $E\left[  0,\infty\right)  $ denote the space of exponential type
functions on $\left[  0,\infty\right)  $ which satisfy the property $\ $
$\left\vert f(x)\right\vert \leq\beta e^{\alpha x}$ for some finite constants
$\alpha,\beta>0.$

In \cite{jaki}, the authors considered the operator $P_{n}$, with%
\begin{equation}
P_{n}(f;x)=\frac{e^{-nx}}{g(1)}\sum\limits_{k=0}^{\infty}p_{k}(nx)f(\frac
{k}{n})\label{1.2}%
\end{equation}
for $f\in E\left[  0,\infty\right)  $ and studied approximation properties of
these operators, as well as the analogue to Sz\`{a}sz's results.

\bigskip If $g(u)\equiv1,$from (\ref{1.1}) we obtain $p_{k}(x)=\dfrac{x^{k}%
}{k!}$and we obtain classical Sz\`{a}sz-Mirakjan operator which is given by%

\[
S_{n}(f;x)=e^{-nx}%
{\displaystyle\sum\limits_{k=0}^{\infty}}
\dfrac{(nx)^{k}}{k!}f(\frac{k}{n})\text{.}%
\]

In 1969, Wood \cite{wood} showed that the operators $P_{n}$ are positive if
and only if $\frac{a_{k}}{g(1)}\geq0,$ $(k=0,1,...).$In 1996, Ciupa \cite{C1}
was studied the rate of convergence of these operators. In 1999, Abel and Ivan
\cite{Abel} showed an asimptotic expansion of the operators given by
(\ref{1.2}) and their derivatives. In 2003, \.{I}spir \cite{nurh} showed that
the approximation of continuous functions having polynomial growth at infinity
by the operator in (\ref{1.2}). In 2007, Ciupa \cite{C2} defined Modified
Jakimovski-Leviatan operators and studied rate of convergence,order of
approximation and Voronovskaya type theorem. Recently, B\"{u}y\"{u}kyaz\i
c\i\ and et. al, \cite{ibrahim} studied approximation properties of Chlodowsky
type Jakimovski -Leviatan operators. They proved Voronovskaya-type theorem and
studied the convergence of these operators in a weighted space by using a new
type of weighted modulus of continuity.

In this paper, we consider the following Kantorovich type generalization of
Jakimovski-Leviatan operators given by
\begin{equation}
L_{n}^{\ast}(f;x)=\frac{e^{-nx}}{g(1)}\frac{n}{b_{n}}\sum\limits_{k=0}%
^{\infty}p_{k}(\frac{n}{b_{n}}x)\int\limits_{\frac{k}{n}b_{n}}^{\frac{k+1}%
{n}b_{n}}f(t)dx\label{1.3}%
\end{equation}
with $b_{n}$ a positive increasing sequence with the properties%
\begin{equation}
\underset{n\rightarrow\infty}{\lim}b_{n}=\infty,\text{ }\underset{n\rightarrow
\infty}{\lim}\frac{b_{n}}{n}=0\label{limbn}%
\end{equation}
and $p_{k}$ are Appell polynomials defined by (\ref{1.1}). Recently, some
generalizations of (\ref{1.2}) operators have been considered in
\cite{C1},\cite{C2} and \cite{Alto}.

\section{Some approximation properties of $L_{n}^{\ast}(f;x)$}

In approximation theory, the positive approximation processes discovered by
Korovkin play a central role and arise\ in a natural way in many problems
connected with functional analysis, harmonic analysis, measure theory, partial
differential equations and probability theory.

Let $C_{E}\left[  0,\infty\right)  $ denote the set of all continuous
functions $f$ on $\left[  0,\infty\right)  $ with the property that
$\left\vert f(x)\right\vert \leq\beta e^{\alpha x}$ for all $x\geq0$ and some
positive finite $\alpha$ and $\beta$. For a fixed $r\in%
\mathbb{N}
,$ we denote by%
\[
C_{E}^{r}\left[  0,\infty\right)  =\left\{  f\in C_{E}\left[  0,\infty\right)
:f^{\prime},f^{\prime\prime},...f^{\left(  r\right)  }\in C_{E}\left[
0,\infty\right)  \right\}  .
\]

\begin{lemma}
The operators $L_{n}^{\ast}(f;x)$ defined by (\ref{1.3}) satisfy the following equalities.
\end{lemma}%

\begin{equation}
L_{n}^{\ast}(1;x)=1,\label{1.4}%
\end{equation}

\begin{equation}
L_{n}^{\ast}(t;x)=x+\frac{g^{\prime}(1)}{g(1)}\frac{b_{n}}{n}+\frac{b_{n}}%
{n},\label{1.5}%
\end{equation}

\begin{equation}
L_{n}^{\ast}(t^{2};x)=x^{2}+\frac{b_{n}}{n}x\left(  \frac{g(1)+2g^{\prime}%
(1)}{g(1)}+1\right)  +\frac{b_{n}^{2}}{n^{2}}\left(  \frac{2g^{\prime}%
(1)}{g(1)}+\frac{g^{\prime\prime}(1)}{g(1)}+\frac{1}{3}\right)  .\label{1.7}%
\end{equation}

\begin{lemma}
\bigskip The central moments of the operators $L_{n}^{\ast}(f;x)$ are given by%
\[
L_{n}^{\ast}(t-x;x)=\frac{g^{\prime}(1)}{g(1)}\frac{b_{n}}{n}+\frac{b_{n}}{n}%
\]

\end{lemma}%

\begin{equation}
L_{n}^{\ast}((t-x)^{2};x)=\frac{b_{n}^{2}}{n^{2}}\left(  \frac{2g^{\prime}%
(1)}{g(1)}+\frac{g^{\prime\prime}(1)}{g(1)}+1\right) \label{1.8}%
\end{equation}

\begin{theorem}
If $f\in C_{E}\left[  0,\infty\right)  ,$ then $\lim L_{n}^{\ast}(f)=f$
uniformly on $\left[  0,a\right]  $.

\begin{proof}
From (\ref{1.4})-(\ref{1.7}), we have
\[
\lim L_{n}^{\ast}(e_{i};x)=e_{i}\left(  x\right)  ,\text{ \ }i\in\left\{
0,1,2\right\}  ,
\]
where $e_{i}\left(  x\right)  =t^{i}$ . Applying the Korovkin theorem
\cite{Gad2}, we obtain the desired result.
\end{proof}
\end{theorem}

In this section, we deal with the rate of convergence of the $L_{n}^{\ast
}(f;x)$ to $f$ by means of a classical approach, the second modulus of
continuity, and Peetre's K-functional.

Let $f\in\widetilde{C}[0,\infty).$ If $\delta>0$, the modulus of continuity of
$f$ is defined by%
\[
\omega(f,\delta)=\sup_{\underset{\left\vert x-y\right\vert \leq\delta
}{x,y\in\lbrack0,\infty)}}\left\vert f(x)-f(y)\right\vert ,
\]
where $\widetilde{C}[0,\infty)$ denotes the space of uniformly continuous
functions on $[0,\infty)$ . It is also well known that, for any $\delta>0$ and
each $x\in\lbrack0,\infty)$ ,%
\[
\left\vert f(x)-f(y)\right\vert \leq\omega(f,\delta)\left(  \frac{\left\vert
x-y\right\vert }{\delta}+1\right)  .
\]
The next result gives the rate of convergence of the sequence $L_{n}^{\ast
}(f;x)$ to $f$ \ by means of the modulus of continuity.

\begin{theorem}
If $f\in C_{E}\left[  0,\infty\right)  ,$ then for any $x\in\lbrack0,a]$ we
have
\[
\left\vert L_{n}^{\ast}(f;x)-f\left(  x\right)  \right\vert \leq\left\{
1+\frac{1}{\delta}\left(  \sqrt{\theta_{n}}\right)  \right\}  \omega
(f,\delta).
\]
where%
\[
\theta_{n}=\frac{b_{n}^{2}}{n^{2}}\left(  \frac{2g^{\prime}(1)}{g(1)}%
+\frac{g^{\prime\prime}(1)}{g(1)}+1\right)  .
\]

\end{theorem}

\begin{proof}%
\begin{align*}
\left\vert L_{n}^{\ast}(f;x)-f\left(  x\right)  \right\vert  & \leq
\frac{e^{-nx}}{g(1)}\frac{n}{b_{n}}\sum\limits_{k=0}^{\infty}p_{k}(\frac
{n}{b_{n}}x)\int\limits_{\frac{k}{n}b_{n}}^{\frac{k+1}{n}b_{n}}\left\vert
f(s)-f\left(  x\right)  \right\vert ds\\
& \leq\frac{e^{-nx}}{g(1)}\frac{n}{b_{n}}\sum\limits_{k=0}^{\infty}p_{k}%
(\frac{n}{b_{n}}x)\int\limits_{\frac{k}{n}b_{n}}^{\frac{k+1}{n}b_{n}}%
\omega(f,\delta)\left(  \frac{\left\vert s-x\right\vert }{\delta}+1\right)
ds\\
& \leq\left\{  1+\frac{e^{-nx}}{g(1)}\frac{n}{b_{n}}\sum\limits_{k=0}^{\infty
}p_{k}(\frac{n}{b_{n}}x)\int\limits_{\frac{k}{n}b_{n}}^{\frac{k+1}{n}b_{n}%
}\left\vert s-x\right\vert ds\right\}  \omega(f,\delta)
\end{align*}

By using the Cauchy-Schwarz inequality for integration, we get%
\[
\int\limits_{\frac{k}{n}b_{n}}^{\frac{k+1}{n}b_{n}}\left\vert s-x\right\vert
ds\leq\frac{1}{\sqrt{n}}\left(  \int\limits_{\frac{k}{n}b_{n}}^{\frac{k+1}%
{n}b_{n}}\left\vert s-x\right\vert ^{2}ds\right)  ^{1/2}%
\]
which holds that%
\[
\sum\limits_{k=0}^{\infty}p_{k}(\frac{n}{b_{n}}x)\int\limits_{\frac{k}{n}%
b_{n}}^{\frac{k+1}{n}b_{n}}\left\vert s-x\right\vert ds\leq\frac{1}{\sqrt{n}%
}\sum\limits_{k=0}^{\infty}p_{k}(\frac{n}{b_{n}}x)\left(  \int\limits_{\frac
{k}{n}b_{n}}^{\frac{k+1}{n}b_{n}}\left\vert s-x\right\vert ^{2}ds\right)
^{1/2}.
\]
If we apply the Cauchy-Schwarz inequality, we get%
\begin{align*}
\left\vert L_{n}^{\ast}(f;x)-f\left(  x\right)  \right\vert  & \leq\left\{
1+\frac{1}{\delta}\left(  \frac{e^{-nx}}{g(1)}\frac{n}{b_{n}}\sum
\limits_{k=0}^{\infty}p_{k}(\frac{n}{b_{n}}x)\int\limits_{\frac{k}{n}b_{n}%
}^{\frac{k+1}{n}b_{n}}\left\vert s-x\right\vert ^{2}ds\right)  \right\}
\omega(f,\delta)\\
& =\left\{  1+\frac{1}{\delta}\left(  \sqrt{L_{n}^{\ast}((s-x)^{2};x)}\right)
\right\}  \omega(f,\delta)\\
& \leq\left\{  1+\frac{1}{\delta}\left(  \sqrt{\theta_{n}}\right)  \right\}
\omega(f,\delta).
\end{align*}
Now if we choose $\delta=\sqrt{\theta_{n}},$ it completes the proof.
\end{proof}

Now we remember the second modulus of continuity of $f\in C_{B}[0,\infty$ $) $
which is defined by%
\[
\omega_{2}\left(  f;\delta\right)  =\sup_{0<t<\delta}\left\Vert f\left(
.+2t\right)  -2f\left(  .+t\right)  +f\left(  .\right)  \right\Vert
_{C_{B}^{^{\prime}}}%
\]
where $C_{B}[0,\infty$ ) is the class of real valued functions defined on
$[0,\infty$ $)$ which are bounded and uniformly continuous with the norm
$\left\Vert f\right\Vert _{C_{B}}=\sup_{x\in\lbrack0,\infty)}\left\vert
f\left(  x\right)  \right\vert .$

Peetre's $K-$functional of the function $f\in C_{B}[0,\infty$ $)$ is defined
by%
\begin{equation}
K(f;\delta)=\inf\left\{  \left\Vert f-g\right\Vert _{C_{B}}+\delta\left\Vert
g\right\Vert _{C_{B}^{2}}\right\}  ,\label{Petree}%
\end{equation}
where
\[
C_{B}^{2}[0,\infty)=\left\{  g\in C_{B}[0,\infty):g^{^{\prime}},g^{^{\prime
\prime}}\in C_{B}[0,\infty)\right\}  ,
\]
and the norm $\left\Vert g\right\Vert _{C_{B}^{2}}=\left\Vert g\right\Vert
_{C_{B}^{2}}+\left\Vert g^{^{\prime}}\right\Vert _{C_{B}}+\left\Vert
g^{^{\prime\prime}}\right\Vert _{C_{B}}$. \ The following inequality
\begin{equation}
K(f;\delta)\leq M\left\{  \omega_{2}\left(  f;\sqrt{\delta}\right)
+\min\left(  1,\delta\right)  \left\Vert f\right\Vert _{C_{B}}\right\}
,\label{W2}%
\end{equation}
holds for all $\delta>0$ and the constant $M$ is independent of $f$ and
$\delta$.

\begin{theorem}
If $f\in$ $C_{B}^{2}[0,\infty)$, then we have%
\[
\left\vert L_{n}^{\ast}(f;x)-f\left(  x\right)  \right\vert \leq\xi\left\Vert
f\right\Vert _{C_{B}^{2}},
\]
where
\end{theorem}%

\[
\xi:=\xi_{n}\left(  x\right)  =\left\{  \frac{g^{\prime}(1)}{g(1)}\frac{b_{n}%
}{n}+\frac{b_{n}}{n}+\frac{b_{n}^{2}}{n^{2}}\left(  \frac{2g^{\prime}%
(1)}{g(1)}+\frac{g^{\prime\prime}(1)}{g(1)}+1\right)  \right\}  .
\]

\begin{proof}
From the Taylor expansion of $f$, the linearity of the operators $L_{n}^{\ast
}$ and (\ref{1.4}), we have%
\begin{equation}
L_{n}^{\ast}(f;x)-f\left(  x\right)  =f^{\prime}\left(  x\right)  L_{n}^{\ast
}(s-x;x)+\frac{1}{2}f^{\prime\prime}\left(  \eta\right)  L_{n}^{\ast}(\left(
s-x\right)  ^{2};x),\eta\in\left(  x,s\right)  .\label{Taylor}%
\end{equation}
Since
\[
L_{n}^{\ast}(s-x;x)=\frac{g^{\prime}(1)}{g(1)}\frac{b_{n}}{n}+\frac{b_{n}}%
{n}\geq0
\]
for $s\geq x$, by considering Lemma and (\ref{Taylor}), we can write
\begin{align*}
\left\vert L_{n}^{\ast}(f;x)-f\left(  x\right)  \right\vert  & \leq\left\{
\frac{g^{\prime}(1)}{g(1)}\frac{b_{n}}{n}+\frac{b_{n}}{n}\right\}  \left\Vert
f^{\prime}\right\Vert _{C_{B}}+\left\{  \frac{b_{n}^{2}}{n^{2}}\left(
\frac{2g^{\prime}(1)}{g(1)}+\frac{g^{\prime\prime}(1)}{g(1)}+1\right)
\right\}  \left\Vert f^{^{\prime\prime}}\right\Vert _{C_{B}}\\
& \leq\left\{  \frac{g^{\prime}(1)}{g(1)}\frac{b_{n}}{n}+\frac{b_{n}}{n}%
+\frac{b_{n}^{2}}{n^{2}}\left(  \frac{2g^{\prime}(1)}{g(1)}+\frac
{g^{\prime\prime}(1)}{g(1)}+1\right)  \right\}  \left\Vert f\right\Vert
_{C_{B}^{2}}.
\end{align*}
which completes the proof.
\end{proof}

\begin{theorem}
Let $f\in C_{B}[0,\infty)$, then
\[
\left\vert L_{n}^{\ast}(f;x)-f\left(  x\right)  \right\vert \leq2M\left\{
\omega_{2}\left(  f;\sqrt{\delta}\right)  +\min\left(  1,\delta\right)
\left\Vert f\right\Vert _{C_{B}}\right\}  ,
\]
where%
\[
\delta:=\delta_{n}\left(  x\right)  =\frac{1}{2}\varsigma_{n}\left(  x\right)
\]
and $M>0$ is a constant which is independent of the functions $f$ and $\delta
$. Also, $\varsigma_{n}\left(  x\right)  $ is the same as in the Theorem 3.
\end{theorem}

\begin{proof}
Suppose that $g\in C_{B}^{2}[0,\infty)$. From Theorem 3, we can write%
\begin{align}
\left\vert L_{n}^{\ast}(f;x)-f\left(  x\right)  \right\vert  & \leq\left\vert
L_{n}^{\ast}(f-g;x)\right\vert +\left\vert L_{n}^{\ast}(g;x)-g\left(
x\right)  \right\vert +\left\vert g(x)-f\left(  x\right)  \right\vert
\label{Esz}\\
& \leq2\left\Vert f-g\right\Vert _{C_{B}}+\varsigma\left\Vert g\right\Vert
_{C_{B}^{2}}\nonumber\\
& =2\left[  \left\Vert f-g\right\Vert _{C_{B}}+\delta\left\Vert g\right\Vert
_{C_{B}^{2}}\right] \nonumber
\end{align}
The left-hand side of inequality (\ref{Esz}) does not depend on the function
$g\in$ $C_{B}^{2}[0,\infty)$, so
\[
\left\vert L_{n}^{\ast}(f;x)-f\left(  x\right)  \right\vert \leq2K(f,\delta)
\]
holds where $K$ $\left(  f;\delta\right)  $ is Peetre's $K-$functional defined
by (\ref{Petree}) . By the relation between Peetre's K functional and the
second modulus of smoothness given by (\ref{W2}) , inequality (\ref{Esz})
becomes%
\[
\left\vert L_{n}^{\ast}(f;x)-f\left(  x\right)  \right\vert \leq2M\left\{
\omega_{2}\left(  f;\sqrt{\delta}\right)  +\min\left(  1,\delta\right)
\left\Vert f\right\Vert _{C_{B}}\right\}
\]
hence we have the result.
\end{proof}

Now let us consider the Lipschitz type space with two parameters (see
\cite{M.Ali}).%
\[
Lip_{M}^{\left(  \alpha_{1},\alpha_{2}\right)  }\left(  \alpha\right)
:=\left\{  f\in C_{B}\left[  0,\infty\right)  :\left\vert f(t)-f\left(
x\right)  \right\vert \leq M\frac{\left\vert t-x\right\vert ^{\alpha}}{\left(
t+\alpha_{1}x^{2}+\alpha_{2}x\right)  ^{\frac{\alpha}{2}}};x,t\in\left[
0,\infty\right)  \right\}
\]
for $\alpha_{1},\alpha_{2}>0,$ $M$ is a positive constant and $\alpha
\in\left[  0,1\right)  .$

\begin{theorem}
Let $f\in Lip_{M}^{\left(  \alpha_{1},\alpha_{2}\right)  }\left(
\alpha\right)  .$ For all $x>0$, we have%
\[
\left\vert L_{n}^{\ast}(f;x)-f\left(  x\right)  \right\vert \leq M\left(
\frac{\frac{b_{n}^{2}}{n^{2}}\left(  \frac{2g^{\prime}(1)}{g(1)}%
+\frac{g^{\prime\prime}(1)}{g(1)}+1\right)  }{\alpha_{1}x^{2}+\alpha_{2}%
x}\right)  ^{\frac{\alpha}{2}}%
\]

\end{theorem}

\begin{proof}
Let $\alpha=1.$%
\begin{align*}
\left\vert L_{n}^{\alpha,\ast}(f;x)-f\left(  x\right)  \right\vert  &
\leq\frac{e^{-nx}}{g(1)}\frac{n}{b_{n}}\sum\limits_{k=0}^{\infty}p_{k}%
(\frac{n}{b_{n}}x)\int\limits_{\frac{k}{n}b_{n}}^{\frac{k+1}{n}b_{n}%
}\left\vert f(t)-f\left(  x\right)  \right\vert dt\\
& \leq M\frac{e^{-nx}}{g(1)}\frac{n}{b_{n}}\sum\limits_{k=0}^{\infty}%
p_{k}(\frac{n}{b_{n}}x)\int\limits_{\frac{k}{n}b_{n}}^{\frac{k+1}{n}b_{n}%
}\left\vert t-x\right\vert dt\\
& \leq\frac{M}{\sqrt{\alpha_{1}x^{2}+\alpha_{2}x}}L_{n}^{\alpha,\ast
}(\left\vert t-x\right\vert ;x)\\
& \leq M\left(  \sqrt{\frac{L_{n}^{\alpha,\ast}((t-x)^{2};x)}{\alpha_{1}%
x^{2}+\alpha_{2}x}}\right) \\
& \leq M\frac{b_{n}}{n}\sqrt{\frac{\frac{2g^{\prime}(1)}{g(1)}+\frac
{g^{\prime\prime}(1)}{g(1)}+1}{\alpha_{1}x^{2}+\alpha_{2}x}}.
\end{align*}
Let $\alpha\in\left(  0,1\right)  .$By applying H\"{o}lder inequality with
$p=\frac{1}{\alpha}$ and $q=\frac{1}{1-\alpha}$%
\begin{align*}
\left\vert L_{n}^{\alpha,\ast}(f;x)-f\left(  x\right)  \right\vert  &
\leq\frac{e^{-nx}}{g(1)}\frac{n}{b_{n}}\sum\limits_{k=0}^{\infty}p_{k}%
(\frac{n}{b_{n}}x)\int\limits_{\frac{k}{n}b_{n}}^{\frac{k+1}{n}b_{n}%
}\left\vert f(t)-f\left(  x\right)  \right\vert dt\\
& \leq\left\{  \frac{e^{-nx}}{g(1)}\sum\limits_{k=0}^{\infty}p_{k}(\frac
{n}{b_{n}}x)\left(  \frac{n}{b_{n}}\int\limits_{\frac{k}{n}b_{n}}^{\frac
{k+1}{n}b_{n}}\left\vert f(t)-f\left(  x\right)  \right\vert dt\right)
^{\frac{1}{\alpha}}\right\}  ^{\alpha}\\
& \leq\left\{  \frac{n}{b_{n}}\frac{e^{-nx}}{g(1)}\sum\limits_{k=0}^{\infty
}p_{k}(\frac{n}{b_{n}}x)\int\limits_{\frac{k}{n}b_{n}}^{\frac{k+1}{n}b_{n}%
}\left\vert f(t)-f\left(  x\right)  \right\vert ^{\frac{1}{\alpha}}dt\right\}
^{\alpha}\\
& \leq M\left\{  \frac{n}{b_{n}}\frac{e^{-nx}}{g(1)}\sum\limits_{k=0}^{\infty
}p_{k}(\frac{n}{b_{n}}x)\int\limits_{\frac{k}{n}b_{n}}^{\frac{k+1}{n}b_{n}%
}\frac{\left\vert t-x\right\vert }{\sqrt{t+\alpha_{1}x^{2}+\alpha_{2}x}%
}dt\right\}  ^{\alpha}\\
& \leq\frac{M}{\left(  \alpha_{1}x^{2}+\alpha_{2}x\right)  ^{\frac{\alpha}{2}%
}}\left\{  \frac{n}{b_{n}}\frac{e^{-nx}}{g(1)}\sum\limits_{k=0}^{\infty}%
p_{k}(\frac{n}{b_{n}}x)\int\limits_{\frac{k}{n}b_{n}}^{\frac{k+1}{n}b_{n}%
}\left\vert t-x\right\vert dt\right\}  ^{\alpha}\\
& \leq\frac{M}{\left(  \alpha_{1}x^{2}+\alpha_{2}x\right)  ^{\frac{\alpha}{2}%
}}\left(  L_{n}^{\alpha,\ast}(\left\vert t-x\right\vert ;x)\right)  ^{\alpha
}\\
& \leq M\left(  \frac{L_{n}^{\alpha,\ast}((t-x)^{2};x)}{\alpha_{1}x^{2}%
+\alpha_{2}x}\right)  ^{\frac{\alpha}{2}}=M\left(  \frac{\frac{b_{n}^{2}%
}{n^{2}}\left(  \frac{2g^{\prime}(1)}{g(1)}+\frac{g^{\prime\prime}(1)}%
{g(1)}+1\right)  }{\alpha_{1}x^{2}+\alpha_{2}x}\right)  ^{\frac{\alpha}{2}}%
\end{align*}

\end{proof}

\section{\bigskip Approximation properties in weighted spaces}

Now we present the weighted spaces of functions that appear in the paper. With
this purpose we firstly introduce the function

We give approximation properties of the operators $L_{n}^{\ast}$ of the
weighted spaces of continuous functions with exponential growth on $R_{0}%
^{+}=\left[  0,\infty\right)  $ with the help of the weighted Korovkin type
theorem proved by Gadjiev in \cite{Gad 1}, \cite{Gad2}. Therefore we consider
the following weighted spaces of functions which are defined on the $R_{0}%
^{+}.$

Let $\rho(x)$ be the weight function and $M_{f}$ be a positive constant, we define%

\[
B_{\rho}\left(  R_{0}^{+}\right)  =\left\{  f\in E\left(  R_{0}^{+}\right)
:\left\vert f(x)\right\vert \leq M_{f}\rho(x)\right\}  ,
\]%
\[
C_{\rho}\left(  R_{0}^{+}\right)  =\left\{  f\in B\left(  R_{0}^{+}\right)
:f\text{ is continuous}\right\}  ,
\]%
\[
C_{\rho}^{k}\left(  R_{0}^{+}\right)  =\left\{  f\in C\left(  R_{0}%
^{+}\right)  :\underset{x\rightarrow\infty}{\lim}\frac{f(x)}{\rho(x)}%
=K_{f}<\infty\right\}  .
\]
It is obvious that $C_{\rho}^{k}\left(  R_{0}^{+}\right)  \subset C_{\rho
}\left(  R_{0}^{+}\right)  \subset B_{\rho}\left(  R_{0}^{+}\right)  $.
$B_{\rho}\left(  R_{0}^{+}\right)  $ is a linear normed space with the norm%
\[
\left\Vert f\right\Vert _{\rho}=\underset{x\in R_{0}^{+}}{\sup}\frac
{\left\vert f(x)\right\vert }{\rho(x)}.
\]

The following results on the sequence of positive linear operators in these
spaces are given in \cite{Gad 1}, \cite{Gad2}.

\begin{lemma}
The sequence of positive linear operators $\left(  L_{n}\right)  _{n\geq1}$
which act from $C_{\rho}\left(  R_{0}^{+}\right)  $ to $B_{\rho}\left(
R_{0}^{+}\right)  $ if and only if there exists a positive constant $k$ such
that $L_{n}(\rho;x)\leq k\rho(x),$ $i.e.$ $\left\Vert L_{n}(\rho;x)\right\Vert
_{\rho}\leq k$.
\end{lemma}

\begin{theorem}
Let $\left(  L_{n}\right)  _{n\geq1}$ be the space of positive linear
operators which act from $C_{\rho}\left(  R_{0}^{+}\right)  $ to $B_{\rho
}\left(  R_{0}^{+}\right)  $ satisfying the conditions%
\[
\underset{n\rightarrow\infty}{\lim}\left\Vert L_{n}(t^{v};x)-x^{v}\right\Vert
_{\rho}=0,\text{ }v=0,1,2,
\]
then for any function $f\in C_{\rho}^{k}\left(  R_{0}^{+}\right)  $%
\[
\underset{n\rightarrow\infty}{\lim}\left\Vert L_{n}f-f\right\Vert _{\rho}=0.
\]

\end{theorem}

\bigskip

\begin{lemma}
Let $\rho(x)=1+x^{2}$ be a weight function. If $f\in C_{\rho}\left(  R_{0}%
^{+}\right)  $, then%
\[
\left\Vert L_{n}^{\ast}(\rho;x)\right\Vert _{\rho}\leq1+M.
\]

\end{lemma}

\begin{proof}
Using (\ref{1.4}) and (\ref{1.5}), we have%
\[
L_{n}^{\ast}(\rho;x)=1+x^{2}+\frac{b_{n}}{n}x\left(  \frac{g(1)+2g^{\prime
}(1)}{g(1)}+1\right)  +\frac{2b_{n}^{2}}{n^{2}}\frac{g^{\prime}(1)}%
{g(1)}+\frac{b_{n}^{2}}{n^{2}}\frac{g^{\prime\prime}(1)}{g(1)}+\frac{b_{n}%
^{2}}{3n^{2}}%
\]%
\begin{align*}
\left\Vert L_{n}^{\ast}(\rho;x)\right\Vert _{\rho}  & =\underset{x\in
R_{0}^{+}}{\sup}\frac{1}{1+x^{2}}\left[  1+x^{2}+\frac{b_{n}}{n}x\left(
\frac{g(1)+2g^{\prime}(1)}{g(1)}+1\right)  +\frac{2b_{n}^{2}}{n^{2}}%
\frac{g^{\prime}(1)}{g(1)}+\frac{b_{n}^{2}}{n^{2}}\frac{g^{\prime\prime}%
(1)}{g(1)}+\frac{b_{n}^{2}}{3n^{2}}\right] \\
& \leq1+\frac{b_{n}}{n}\left(  \frac{g(1)+2g^{\prime}(1)}{g(1)}+1\right)
+\frac{2b_{n}^{2}}{n^{2}}\frac{g^{\prime}(1)}{g(1)}+\frac{b_{n}^{2}}{n^{2}%
}\frac{g^{\prime\prime}(1)}{g(1)}+\frac{b_{n}^{2}}{3n^{2}}%
\end{align*}
since $\underset{n\rightarrow\infty}{\lim}\frac{b_{n}}{n}=0$, there exists a
positive constant $M$ such that%
\[
\left\Vert L_{n}^{\ast}(\rho;x)\right\Vert _{\rho}\leq1+M.
\]

\end{proof}

By using Lemma 3, we can easily see that the operators $L_{n}^{\ast}$ defined
by (\ref{1.3}) act from $C_{\rho}\left(  R_{0}^{+}\right)  $ to $B_{\rho
}\left(  R_{0}^{+}\right)  $.

\begin{theorem}
Let $L_{n}^{\ast}$ be the sequence of linear positive operators defined by
(\ref{1.3}) and $\rho(x)=1+x^{2}$, then for each $f\in C_{\rho}^{k}\left(
R_{0}^{+}\right)  $%
\[
\underset{n\rightarrow\infty}{\lim}\left\Vert L_{n}^{\ast}(\rho
;x)-f(x)\right\Vert _{\rho}=0.
\]

\begin{proof}
It is enough to show that the conditions of the weighted Korovkin type theorem
given by Theorem 6. From (\ref{1.4}), we can write%
\begin{equation}
\underset{n\rightarrow\infty}{\lim}\left\Vert L_{n}^{\ast}(1;x)-1\right\Vert
_{\rho}=0.\label{2.1}%
\end{equation}
Using (\ref{1.5}), we have%
\[
\left\Vert L_{n}^{\ast}(e_{1};x)-e_{1}(x)\right\Vert _{\rho}=\frac{b_{n}}%
{n}\frac{g^{\prime}(1)}{g(1)}+\frac{b_{n}}{n}%
\]
this implies that%
\begin{equation}
\underset{n\rightarrow\infty}{\lim}\left\Vert L_{n}^{\ast}(e_{1}%
;x)-e_{1}(x)\right\Vert _{\rho}=0.\label{2.2}%
\end{equation}
From (\ref{1.7}),%
\begin{align*}
\left\Vert L_{n}^{\ast}(e_{2};x)-e_{2}(x)\right\Vert _{\rho}  &
=\underset{x\in R_{0}^{+}}{\sup}\frac{1}{1+x^{2}}\left\{  \frac{b_{n}}%
{n}x\left(  \frac{g(1)+2g^{\prime}(1)}{g(1)}+1\right)  +\frac{b_{n}^{2}}%
{n^{2}}\left(  \frac{2g^{\prime}(1)}{g(1)}+\frac{g^{\prime\prime}(1)}%
{g(1)}+\frac{1}{3}\right)  \right\} \\
& \leq\frac{b_{n}}{n}\left(  \frac{g(1)+2g^{\prime}(1)}{g(1)}+1\right)
+\frac{b_{n}^{2}}{n^{2}}\left(  \frac{2g^{\prime}(1)}{g(1)}+\frac
{g^{\prime\prime}(1)}{g(1)}+\frac{1}{3}\right)  .
\end{align*}
Using the conditions (\ref{limbn}), it follows that%
\begin{equation}
\underset{n\rightarrow\infty}{\lim}\left\Vert L_{n}^{\ast}(e_{2}%
;x)-e_{2}(x)\right\Vert _{\rho}=0.\label{2.3}%
\end{equation}
From (\ref{2.1}), (\ref{2.2}) and (\ref{2.3}) for $v=0,1,2,$ we have%
\[
\underset{n\rightarrow\infty}{\lim}\left\Vert L_{n}^{\ast}(e_{v}%
;x)-e_{v}(x)\right\Vert _{\rho}=0.
\]
If we apply Theorem 6, we obtain desired result.
\end{proof}
\end{theorem}

\end{document}